\documentclass{article}

\usepackage[colorlinks, linkcolor=blue, citecolor=blue]{hyperref}           
\usepackage{color}
\usepackage{graphicx,amsmath,amssymb,amsfonts,bm,epsfig,epsf,url,dsfont}
\usepackage{amsthm,mathrsfs}

\newtheorem{thm}{Theorem}[section]
\newtheorem{lem}[thm]{Lemma}

\newtheorem{proposition}[thm]{Proposition}
\newtheorem{remark}[thm]{Remark}
\newtheorem{corollary}[thm]{Corollary}

\newtheorem{conjecture}[thm]{Conjecture}

\usepackage{tikz}
\usepackage{bbm}
\usepackage[small,bf]{caption}
\usepackage[top=1in,bottom=1in,left=1in,right=1in]{geometry}
\usepackage{hyperref}
\usepackage{algorithm}
\usepackage{algpseudocode}
\usepackage{verbatim}
\usepackage[titletoc,toc]{appendix}
\usepackage{subcaption}
\usepackage{booktabs}

\newcommand{\rev}[1]{{\color{black} #1}}

\newcommand{\R}{\mathbb{R}}
\newcommand{\Z}{\mathbb{Z}}

\DeclareMathOperator{\Hom}{Hom}

\DeclareMathOperator{\SO}{SO}
\DeclareMathOperator{\Orth}{O}

\DeclareMathOperator{\Sym}{Sym}

\begin{document}

\title{Orbit recovery for spherical functions}

\author{Tamir Bendory, Dan Edidin, Josh Katz, and Shay Kreymer}

\maketitle

\begin{abstract}
Orbit recovery is a central problem in both mathematics and applied sciences, with important applications to structural biology. This paper focuses on recovering generic orbits of functions on
$\R^{n}$ and the sphere $S^{n-1}$
under the rotation action of $\SO(n)$. 
Specifically, we demonstrate that invariants of degree three (called the bispectrum) suffice to recover generic orbits of functions in finite-dimensional approximations of $\mathbf{L}^2(\R^n)$
obtained by band-limiting the spherical component and discretizing the radial direction. In particular, our main result explicitly bounds the number of samples in the radial direction required for recovery from the degree three invariants. From an application perspective, the most important case is 
$\SO(3)$, which arises in many scientific fields, and in particular, plays a central role in leading structural biology applications such as cryo-electron tomography and cryo-electron microscopy. Our result for $\SO(3)$ states that considering three spherical shells (i.e., samples in the radial direction) is sufficient
to recover generic orbits, which verifies an implicit conjecture made in~\cite{bandeira2017estimation}.
Our proof technique provides an explicit, computationally efficient algorithm to recover the signal by successively solving systems of linear equations. We implemented this algorithm and demonstrated its effectiveness on three protein structures.
\end{abstract}

\section{Introduction}
Orbit recovery problems lie at the intersection of invariant theory and practical scientific applications, in particular in structural biology~\cite{bendory2020single}.  
A prototypical example is the multi-reference alignment (MRA) problem, where the goal is to recover a signal from noisy observations subject to random group actions~\cite{bandeira2014multireference,bandeira2017estimation,bendory2017bispectrum}. Invariant polynomials have emerged as a powerful tool for signal estimation in this setting, as they bypass the need to estimate the unknown group elements, a task that becomes particularly difficult in high-noise regimes. Remarkably, it has been shown that, from a statistical standpoint, this approach is optimal for signal recovery in such regimes~\cite{abbe2018estimation,abbe2018multireference}. Beyond MRA, the problem of constructing polynomials that separate all orbits in a representation has been studied for over a century, with significant advances made in recent years. See for example~\cite{derksen2015computational, draisma2008polar, domokos2017degree}. 
In applications, however, it is typically only necessary to separate orbits almost surely or generically. Recent work has demonstrated
that the number and degrees of polynomials required to separate generic orbits is often much lower than those necessary to recover all orbits~\cite{bandeira2017estimation, edidin2023orbit, blum-smith2024degree, edidin2024reflection}.

The simplified, and most studied MRA model, is when the group of circular shifts $\Z_L$ acts on $\R^L$. In this case, it was shown that a generic signal can be recovered from the third-degree invariant~\cite{bendory2017bispectrum,perry2019sample}. Follow-up papers extended the analyses to more intricate groups and transformations, such as $\SO(2)$~\cite{ma2019heterogeneous,drozatz2025provable}, the dihedral group~\cite{bendory2022dihedral, edidin2024reflection} and dilations~\cite{yin2024bispectrum}; see~\cite{bandeira2017estimation} for the presentation of a wide range of different cases.  

Although recovering an orbit from invariants of degree three would seem to be an intractable computational problem due to the difficult nature of solving systems of polynomial equations, several of the examples that have generic degree three separation also have efficient and provable algorithms. For example, in~\cite{bendory2017bispectrum} a series of efficient algorithms, including frequency marching and semi-definite programming, were designed for the case of $\Z_L$ acting on $\R^L$. A Jennrich-type algorithm for the same case was designed in~\cite{perry2019sample} and this was extended to the regular representation of any finite group in~\cite{edidin2025orbit}. Spectral algorithms were also designed in~\cite{chen2018spectral}; see~\cite{abbe2018multireference,drozatz2025provable,balanov2026projected} for related algorithms. 

In this work, we study representations of $\SO(n)$ of the form 
\begin{align}\label{spherical}
V_L = (\mathbf{L}^2(S^{n-1})_L)^R,
\end{align}
where $\mathbf{L}^2(S^{n-1})_L$ refers to $L$-bandlimited functions on the $(n-1)$-sphere. 
Our goal is to determine conditions that ensure that invariants of degree three can recover the orbit of a generic 
signal in $V_L$. 
Invariants of degree three are also referred to as the bispectrum, and we use the two terms interchangeably.
We follow a frequency marching strategy, first suggested by the authors of~\cite{bandeira2017estimation} and also used in~\cite{liu2021algorithms}. 
To derive the theoretical results, we use a small subset of the degree-three invariants to recursively solve the components of higher frequency from previously solved invariants of lower degree. Precisely, the invariants we consider give rise to linear equations for the component of a signal at higher frequency with coefficients determined by the lower frequency components. 
The main result of this paper, Theorem~\ref{thm.main}, gives a bound on the number of shells $R$ (i.e., samples in the radial direction) as a function of $n$ and $L$ needed to ensure that
we can recover the generic orbit from invariants of degree at most three. The proof is based on showing that the set of linear equations we consider always has full rank. 
The main result is presented in Section~\ref{sec:results} and proved in Sections~\ref{sec:invariant}  and~\ref{sec:proof_main}.

Section~\ref{sec:so3} is devoted to the analysis of the $n=3$ case, the main application of our work. The problem of identifying an unknown signal in $\R^3$ up to rotation 
by an element in $\SO(3)$ occurs in a number of contexts, including two related techniques in molecular imaging, cryo-electron microscopy (cryo-EM)~\cite{milne2013cryo,kam1980reconstruction} and cryo-electron tomography (cryo-ET)~\cite{turk2020promise}. 
For the group $\SO(3)$, our main result, Theorem~\ref{thm.so3} states that $R=3$ shells (independent of $L$) are sufficient to recover the generic signal in $(\mathbf{L}^2(S^2)_L)^3$, for any $L$, from invariants of degree at most three. This affirms an implicit conjecture made in~\cite{bandeira2017estimation} and the bound is known to be optimal. 

In Section~\ref{sec:numerics}, we implement the proposed frequency marching algorithm for $\SO(3)$ to recover the orbits using degree-three invariants. 
Numerically, there is no reason to discard any of the available equations. Leveraging all of them allows us to implement a numerically stable, frequency-marching algorithm that provably recovers the signal by successively solving systems of linear equations. 
We demonstrate the effectiveness, efficiency, and robustness of our approach for \rev{three} molecular structures in the Electron Microscopy Data Bank (EMDB). 

\section{Problem formulation and statement of results} \label{sec:results}

\subsection{Formulation of the problem}
Recall that a polynomial $f \in \R[x_1, \ldots , x_n]$ is {\em harmonic} if $\Delta f = 0$ where $\Delta = \sum_{i=1}^n \frac{\partial^2}{\partial x_i^2}$
is the Laplacian. Since the Laplacian is rotation invariant, the translation of a harmonic polynomial by a rotation
is again a harmonic polynomial. We denote by $H_\ell$ the vector space of homogeneous harmonic polynomials of degree $\ell$ with
the action of $\SO(n)$ by rotation.
Since a homogeneous polynomial is uniquely determined by its values on the sphere, we view elements of $H_\ell$ as spherical harmonic polynomials.

The space of homogeneous polynomials of degree $\ell$ on $S^{n-1}$ decomposes as an $\SO(n)$ representation into irreducible components as
\begin{equation} \label{eq.symdecomp}
\Sym^\ell(S^{n-1})=\bigoplus_{i= 0}^{\lfloor \ell/2 \rfloor} H_{\ell-2i}.
\end{equation}
Equation~\eqref{eq.symdecomp} yields the following formula for the dimension of $H_\ell$:
\begin{equation} \label{eq.harmonicdim}
m_\ell \colon = \dim H_\ell= \binom{n+\ell-1}{n-1}-\binom{n+\ell-3}{n-1}.
\end{equation}
When $n=3$, then $m_\ell = 2\ell+1$ and for general $n$ we have that $m_\ell\sim O(\ell^{n-2})$.

A basic result in harmonic analysis states that we have a decomposition of $\SO(n)$ representations $\mathbf{L}^2(S^{n-1}) = \oplus_{\ell=0}^\infty H_\ell$.
When $n=3$, the representations $H_\ell$ are the usual spaces of spherical harmonics studied in the physics literature; see for example~\cite{knapp2002lie,axler2001harmonic}.
We denote the finite sum $\oplus_{\ell=0}^L H_\ell$ as $\mathbf{L}^2(S^{n-1})_L$ and refer to this finite-dimensional representation
of $\SO(n)$ as the space of $L$-bandlimited functions on the $(n-1)$-sphere. 
In this work, we consider the finite-dimensional $\SO(n)$ representation of the form~\eqref{spherical}, which we view as a discrete approximation for $\mathbf{L}^2(\R^n)$, as a representation of $\SO(n)$. 
Our goal is to determine bounds on $R$---the number of spherical shells---which ensure that the generic signal 
$f \in V_L$ can be recovered from invariants of degree at most three.

\subsection{Main results}
We are now ready to present the main results of this paper. The proofs are provided in Sections~\ref{sec:proof_main} and Section~\ref{sec:so3}, and a technical background is provided in Section~\ref{sec:invariant}.

\begin{thm}\label{thm.main}
Let $V_L= (\mathbf{L}^2(S^{n-1})_L)^R$ be the finite-dimensional approximation of $\mathbf{L}^2(\R^n)$ by functions which are defined
on $R$ spherical shells and $L$-bandlimited on each shell. 
If $R \geq \frac{\rev{m_L} + (\lceil L/2 \rceil -1)}{(L-1)}$,  where $\rev{m_L = \dim H_L}$ is given in~\eqref{eq.harmonicdim}, then the $\Orth(n)$ orbit of a generic real valued $f\in \rev{V_L}$ is determined by $\SO(n)$-invariants of degree at most three.
\end{thm}


In the case of $\SO(3)$, the most important case from the scientific perspective, we can eliminate the reflection ambiguity, and we obtain the following result. Importantly, in this case, only three spherical shells are required, independent of $L$.

\begin{thm}[The $\SO(3)$ case.] \label{thm.so3}
If $R \geq 3$, then for any $L \geq 0$ the invariants of degree at most three separate generic orbits in $(\mathbf{L}^2(S^2)_{L})^R$.
\end{thm}

\rev{\begin{remark}
The constant multiplicity hypothesis in both Theorem~\ref{thm.main} and Theorem~\ref{thm.so3} can be relaxed.
For Theorem~\ref{thm.main}, the conclusion holds provided the number of shells $R_\ell$ in band $\ell$
satisfies $R_\ell \geq \frac{m_\ell + (\lceil \ell/2 \rceil -1)}{(\ell-1)}$. Likewise, the conclusion of
Theorem~\ref{thm.so3} holds provided each shell appears with multiplicity at least three, without requiring the multiplicities to be constant across shells.
\end{remark}}

 Theorem~\ref{thm.so3} improves the linear bound of $R\geq L+2$ given in~\cite{edidin2023orbit} and proves that the 
 hypothesis of~\cite[Theorem 4.19]{bandeira2017estimation} is always satisfied when $R \geq 3$, thereby proving an implicit conjecture made there. Moreover, the bound $R \geq 3$ is optimal since it is known that if $R < 3$ then invariants of degree at most three cannot recover signals with small band limit~\cite{bandeira2017estimation}.
It is an open question as to whether or not there is a constant multiplicity (independent of the band limit) so that degree 3 invariants recover a generic orbit for $\SO(n)$ with $n>3$.

\rev{\begin{remark}Theorems~\ref{thm.main} and~\ref{thm.so3} apply to generic signals. In particular our conclusion
does not a priori hold for 3-dimensional structures with an intrinsic symmetry. However, the frequency marching
algorithm of Section~\ref{sec:numerics} successfully reconstructed two molecular structures with approximate symmetry, TRPV1 ($C_4$ symmetry) and 20S proteasome ($D_7$ symmetry). 
An interesting direction for future work is to prove that the conclusion of Theorem~\ref{thm.so3} also holds for generic structures with symmetry.\end{remark}}

\subsection{Applications to sample complexity ananlys in multi-reference alignment}
The multi-reference alignment (MRA) model entails estimating a signal $f$ in a representation~$V$ of a compact group~$G$ from $m$ observations of the form
\begin{equation} \label{eq:mra}
    y_i = g_i\cdot f + \varepsilon_i, 
\end{equation}
where $g_1,\ldots,g_m\in G$ are random elements drawn from a uniform distribution over the group~$G$, and $\varepsilon_i\sim\mathcal{N}(0,\sigma^2 I)$ are i.i.d. realization of a Gaussian noise with variance $\sigma^2$. 
The goal is to estimate the signal $f$ from $y_1,\ldots,y_m$, while the group elements are treated as nuisance variables. 
Because there is no way to distinguish $f$ from $g\cdot f$, the MRA problem is an orbit recovery problem.
This model was first suggested as an abstraction of the cryo-EM model~\cite{bendory2020single}, and the problem has been studied in more generality as a prototype of statistical models with intrinsic algebraic structures, e.g.,~\cite{bandeira2014multireference,bandeira2017estimation,bendory2024transversality}. 

Sample complexity refers to the number of observations required to accurately estimate a signal. It was shown that the sample complexity of the MRA model, in the high noise regime $\sigma\to\infty,$ is determined by the lowest order moment that identifies the signal~\cite{abbe2018estimation,abbe2018multireference}. If the distribution over the group is uniform, then the moments are equivalent to invariant polynomials.
In particular,  as $m,\sigma\to\infty$, a necessary condition for accurate signal recovery is that $m/\sigma^{2d}\to\infty$, where $d$ is the lowest order moment of the observations that determines the orbit of $f$ uniquely. For example, if the first and second moments do not determine the signal, but the third moment does, then $m$ must scale faster than $\sigma^6$ for accurate recovery. 
Based on this result, beginning in~\cite{bandeira2017estimation}, 
the problem of identifying representations of compact groups for which the third moment can determine generic orbits
has been studied by a number of authors~\cite{blum-smith2024degree,edidin2023orbit, edidin2024reflection, edidin2025orbit, 
blum-smith2025generic}.

Immediate corollaries of Theorem~\ref{thm.main} and Theorem~\ref{thm.so3} are the following: 

\begin{corollary}[The sample complexity of multi-reference alignment for $\SO(n)$] \label{cor:mra_sample_complexity} 
Consider the MRA model~$\eqref{eq:mra}$ with $\sigma\to\infty$, where $V = \left(\oplus_{\ell=0}^L H_\ell\right)^R$ and $G =\SO(n)$.
If $R \geq \frac{\dim H_\ell + (\lceil \ell/2 \rceil -1)}{(\ell-1)}$, then 
the minimal number of observations required for accurate recovery of the $\Orth(n)$ orbit of $f$, regardless of any specific algorithm, is $m/\sigma^6\to\infty$.
\end{corollary}

\begin{corollary}[The sample complexity of multi-reference alignment for $\SO(3)$] \label{cor:mra_sample_complexity_so3} 
Consider the MRA model~$\eqref{eq:mra}$ with $\sigma\to\infty$, where $V = \left(\oplus_{\ell=0}^L H_\ell\right)^R$ and $G =\SO(3)$.
If $R \geq 3$, then the minimal number of observations required for accurate recovery of the $\SO(3)$ orbit of $f$, regardless of any specific algorithm, is $m/\sigma^6\to\infty$.
\end{corollary}

An interesting extension of the MRA model is the heterogeneous MRA model, where each measurement corresponds to a translated and noisy observation of one of 
$K$ underlying signals (the model in~\eqref{eq:mra} corresponds to the special case  ($K=1$)).
In this setting, it is not possible to recover the third-degree invariant of each individual signal; only their average can be estimated. Nevertheless, recent works~\cite{bandeira2017estimation, boumal2018heterogeneous} provide strong evidence that, in various MRA settings, it is possible to estimate the orbits of multiple signals simultaneously $(K>1$) from the average of their degree-three invariants.
We conjecture that a similar phenomenon holds in the model studied in this paper. In particular, a compelling direction for future work is to establish bounds on orbit recovery from third-degree invariants in the heterogeneous case, as a function of the number of shells $R$ and the number of signals $K$.

\subsection{Implications to structural biology.}
The MRA model~\eqref{eq:mra} with $G=\SO(3)$ fits the problem of subtomogram averaging in cryo-ET, which is one of the main steps in recovering molecular structures in in-situ environments~\cite{watson2024advances}. Since the noise level in cryo-ET data is extremely high, Corollary~\ref{cor:mra_sample_complexity_so3} implies that the number of observations must be larger than $\sigma^6$ for recovery. 

The problem of single-particle reconstruction in cryo-EM follows a similar model, but we have access only to the tomographic projections of the rotated copies of the signal~\cite{bendory2020single}. Namely, The cryo-EM model reads
\begin{equation} \label{eq:cryo}
    y_i = P(g_i\cdot f) + \varepsilon_i, 
\end{equation}
where \rev{$Pf(z_1,z_2)=\int_{z_3}f(z_1,z_2,z_3)dz_3$} is a tomographic projection. 
Thus, the results of this paper cannot be directly applied to derive the sample complexity of cryo-EM. 
Yet, following the results of this paper, and~\cite[Conjecture 4.17]{bandeira2017estimation}, we conjecture that for enough spherical shells, a signal can be recovered from the degree-three invariants in the cryo-EM model.

\begin{conjecture}[The sample complexity of the cryo-EM model] \label{conj:cryo} 
Consider the cryo-EM model~$\eqref{eq:cryo}$.
Then, if $R \geq 3$ is large enough, then
the minimal number of observations required for accurate recovery of the $\Orth(n)$ orbit of $f$, regardless of any specific algorithm, is $m/\sigma^6\to\infty$.
\end{conjecture}



\subsection{Related work}
In~\cite[Conjecture 4.11]{bandeira2017estimation},  the authors conjecture that if $L \geq 10$
then the generic orbit in $\mathbf{L}^2(S^2)_L = \oplus_{\ell=0}^L H_\ell$ can be determined up to a finite list of orbits from invariants of 
degree at most three. They also showed that the bound $L \geq 10$ is sharp. This conjecture was recently proved in~\cite{fan2021maximum}. Although \cite[Conjecture 4.11]{bandeira2017estimation} is about finite list recovery, it is possible 
that for $L$ sufficiently large, the generic orbit in $V_L$ can also be recovered from the bispectrum; i.e. an orbit can be determined up to a list of size one. To that end, ~\cite{liu2021algorithms} produces a frequency marching algorithm, which can be used to determine, for $\ell$ sufficiently large, the $\ell$-th frequency component of $f \in V_L$, denoted by $f^\ell$, from a portion of the bispectrum and the prior
knowledge of the components $f^i$ with $i < \ell$. Each step in the frequency march requires the solution to a linear system of equations for coefficients of $f^\ell$, and the condition number of this system can be bounded in terms of the parameters of the model. However, the algorithm of~\cite{liu2021algorithms} cannot be used to directly recover an unknown function $f \in V_L$ from invariants of low degree, since it requires as prior input an unspecified number of components of $f$ of low frequency.

We also mention that when the support of the signal on the sphere is small then it was demonstrated in~\cite{bendory2022compactification} 
that the $\SO(3)$-orbit can also be recovered from the bispectrum.
In addition, a recent series of works showed that recovery from the second moment is possible, if the signal is known to lie in a semi-algebraic set of low dimension; for example, if the signal is sparse or is in the image of a deep neural network~\cite{bendory2022sparse,bendory2024transversality,amir2025stability}.

\section{Representation and invariant theory of harmonic polynomials}
\label{sec:invariant}
The \\
decomposition~\eqref{eq.symdecomp} implies that $H_\ell$ is the irreducible representation of $\SO(n)$ with highest weight vector $\ell L_1$ 
in the notation of~\cite[Section 19]{fulton1991representation}.
As a consequence, we obtain the following proposition.
\begin{proposition}\label{highestweight}
The tensor product $H_i \otimes H_{\ell-i}$ contains a unique copy $H_\ell$.  
\end{proposition}
\begin{proof}
The highest weight vector contained in the tensor product $H_i \otimes H_{\ell-i}$ is, in the notation of~\cite{fulton1991representation}, $i L_1 + (\ell -i)L_1 =\ell L_1$, which appears with multiplicity one. 
Hence, $H_i\otimes H_{\ell-i}$ 
contains a single copy $H_\ell$ as a highest weight representation.
\end{proof}

\subsection{Invariants of $\mathbf{L}^2(S^{n-1})_L$}
We now describe a collection of polynomial invariants of degree at most three in $V_L :=\mathbf{L}^2(S^{n-1})_L$.
To produce these invariant polynomials we first identify invariant tensors in $V_L^{\otimes k}$ (for $1 \leq k \leq 3$)
and then symmetrize\footnote{Recall that if $V$ is a vector space then the symmetrization of a tensor $t = v_1 \otimes \ldots \otimes v_k \in V^{\otimes k}$ is the symmetric tensor $St = {1\over{k!}}\sum_{\sigma \in S_k} (v_\sigma(1) \otimes \ldots \otimes v_{\sigma(k)})$. This operation
can be extended linearly to define a projection $S \colon V^{\otimes k} \to \Sym^k V$. The symmetrization of an
arbitrary element $x \in V^{\otimes k}$ is its image under this projection.} them to produce invariant polynomials in $\Sym^k V_L$. 
To begin, choose an orthonormal basis $e^\ell_1, \ldots , e^\ell_{m_\ell}$ for each summand $H_\ell$ in $V_L$.
An element $f \in V_L$ can be represented as an $L$-tuple $f= (f^0, \ldots , f^L)$ and each $f^\ell$ can be expanded 
as $f^\ell = \sum_{k=1}^{m_\ell} f^\ell_k e^\ell_k$.

\textbf{Invariants of degree one.} If $f \in V_L$ then the invariants of degree one in $f$ is just the projection of $f$ to the trivial one dimensional
summand $H_0$; i.e., the only invariant of degree one is $f^0$.

\textbf{Invariants of degree two.} The representations $H_\ell$ are self-dual via the quadratic form defining $\SO(n)$. Hence 
 $V_L \otimes V_L$ is identified with $\Hom(V_L, V_L)$. By Schur's lemma, $(V_L \otimes V_L)^{\SO(n)} = \oplus_{\ell=0}^L (H_\ell \otimes H_\ell)^\rev{\SO(n)}$
 and $(H_\ell \otimes H_\ell)^\rev{\SO(n)}$ is generated by the symmetric tensor $\sum_{k=1}^{m_\ell}e_k^\ell \otimes e_k^\ell$.
Thus, the invariants of degree 2 in $V_L$ are generated by the quadratic norms; i.e., the invariant polynomials $\sum_{k=1}^{m_\ell} (f_k^\ell)^2$.

\textbf{Invariants of degree three.}
Here, we only construct a subset of the degree three invariants that are sufficient to distinguish generic orbits.
Our approach to constructing invariants uses representation theory. Specifically, we will use representation-theoretic methods to construct invariant tensors in $(V_L)^{\otimes 3}$ and then symmetrize them to obtain an invariant polynomial. 
The following lemma identifies, for each $\ell$, a set of $\lfloor \ell /2\rfloor$ distinct invariant polynomials
of degree three, which are linear in the $\ell$-th component vector of $f \in \mathbf{L}^2(S^{n-1})_\ell$.

\begin{lem}\label{lem.degree3}
For fixed $\ell \leq L$ and for $1 \leq i \leq \lfloor \ell/2 \rfloor$,  the tensor product $H_i \otimes H_{\ell-i} \otimes H_{\ell}$  a nontrivial invariant tensor
whose symmetrization is not identically zero. Hence, we obtain an invariant polynomial of degree 3, which we label as $I_3(i,\ell-i,\ell)\in 
\Sym^3(H_i+H_{\ell-i}+H_\ell)^{\SO(n)}$.
\end{lem}
\begin{proof}
By Proposition~\ref{highestweight}, the tensor product $H_i \otimes H_{\ell-i}$ contains a single copy of the irreducible representation
$H_\ell$. Hence, by Schur's lemma, $(H_i \otimes H_{\ell-i} \otimes H_\ell)^G$ is one-dimensional. To produce an invariant polynomial, let $\phi_\ell \colon H_i \otimes H_{\ell-i} \to H_\ell$ be the projection onto the summand isomorphic to $H_\ell$. If $T_1,\ldots,T_{m_\ell} \in H_i \otimes H_{\ell-i}$ is an orthonormal basis for this summand, which is identified with the prechosen basis, then $\sum_{k=1}^{m_\ell} T_k \otimes e^\ell_k$
is an invariant in $H_i \otimes H_{\ell-i} \otimes H_\ell$. If $i ,\ell-i, \ell$ are distinct, then it is immediate that its symmetrization is non-zero. 
In the case where $i= \ell-i$ we note that the summand in $H_i \otimes H_i$ isomorphic to $H_{2i}$, necessarily lies in $\Sym^2 H_i$. 
The reason is that the highest weight of $H_{2i}$ equals the highest weight in the (reducible) representation $\Sym^2 H_{i}$ which is $(2i)L_1$. 
Thus, in this case, the $T_k$ are already symmetric, so the symmetrization of $\sum_{k=1}^{m_\ell} T_k \otimes e_k^\ell$ also cannot vanish.
\end{proof}

The invariants constructed in Lemma~\ref{lem.degree3} can be made explicit. 
Let $f=(f^1 + \ldots + f^L) \in V_L$ with $f^\ell = \sum_{k=1}^{m_\ell} f^\ell_k e^\ell_k$.
We can write the projection of $f^i \otimes f^{\ell-i}$ to the summand $H_\ell$ as $\sum_{k=1}^{m_\ell}  C_k(f^i,f^{\ell -i}) T_k$, where
$C_k(f^i,f^{\ell-i})$ are bilinear forms in the coefficients $\{f^i_s, f^{\ell-i}_t\}$. Precisely, we have
\begin{equation}
C_k(f^i,f^{\ell-i}) = \sum_{s=1}^{m_i}\sum_{t=1}^{m_{\ell-i}} c_{s,t}^k f^i_s f^{\ell-i}_t.
\end{equation}
Here, the coefficients $c_{s,t}^k$ are the Clebsch-Gordan coefficients for
$\SO(n)$.
Since the projection $H_i \otimes  H_{\ell-i} \to H_\ell$ is surjective, we know that for generic choice of $f_i, f_{\ell-i}$ the coefficients $C_k(f_i,f_{\ell-i})$ are identically non-zero and that there can be no linear relations between 
the $C_k(f^i,f^{\ell-i})$ as bilinear forms in $f^i_s, f^{\ell-i}_t$.
The coefficients of the invariant part of the symmetrization of $f^i \otimes f^{\ell-i} \otimes f^\ell$ are the invariant polynomials
$\sum_{k=1}^{m_\ell} C_k(f^i,f^{\ell-i}) f^\ell_k$, which we denote by $I(i,\ell-i,\ell)(f^i,f^{\ell-i},f^\ell).$

\section{Proof of Theorem \ref{thm.main}}
\label{sec:proof_main}
We view an element $f \in \left(\mathbf{L}^2(S^{n-1})_L\right)^{R}$ as an $R$-tuple of functions $f[1], \ldots , f[R]$, 
one for each of the $R$ shells. Each function $f[r]$ is assumed to be bandlimited to frequency $L$, 
so we can write $f[r]= \sum_{\ell=0}^L f^\ell[r]$. In turn, the function $f^\ell[r]$ can be expanded 
in terms of an orthonormal basis $\{e_k^\ell[r]\}_{k =1}^{m_\ell}$ as $f^\ell[r] = \sum_{k=1}^{m_\ell} f_k^\ell[r]e_k^\ell[r]$.
Following similar notation used in~\cite{bendory2024sample}, we let $A^\ell$ be the $m_\ell \times R$ matrix such that $A^\ell_{k, r} = f_k^\ell[r]$.
Our goal is to recover, for a generic $f$, the matrices $A^\ell$ from invariants of degree at most three. 

Given a set of vectors $\{f^j[m]\}$ with $1 \leq j \leq \lceil \ell/2 \rceil$ and
$1\leq m \leq R$, the invariants $I_3(i,\ell-i,\ell)(f^i[m],f^{\ell-i}[m],f^\ell[r])$ determine linear equations for the coefficients
$f_k^\ell[r]$
\begin{equation}
\sum_{k=1}^{m_\ell} C_k(f^i[m],f^{\ell-i}[n]) f_k^\ell[r] = I_3(i,\ell-i,\ell)(f^i[m],f^{\ell-i}[n],f^\ell[r]).       
\end{equation}
If $i \neq j$, let 
\begin{align*}
 S_{i,j}[R] &= \{ (1,1), \ldots, (1,R), (2,1), \ldots , (R,1)\} \subset [1,R]^2, 
\end{align*}
and let 
\begin{align*}
S_{i,i}[R] &= \{(1,1), \ldots, (1,R)\}.
\end{align*}
Consider the system of $R(\ell-1) - (\lceil \ell/2 \rceil -1)$ linear equations 
for the coefficients $f^\ell[r]_k$
\begin{equation} \label{eq.system}
\left\{ \sum_{k=1}^{m_\ell} C_k(f^i[m],f^{\ell-i}[n]) f_k^\ell[r] = I_3(i,\ell-i,\ell)(f^i[m],f^{\ell-i}[n],f^\ell[r]) 
\right\}_{1 \leq i \leq \lfloor \ell/2 \rfloor, (m,n) \in S_{i,\ell-i}}
\end{equation}
\begin{proposition}\label{prop.freqmarching}
For a generic choice $f^j[m]$ with $j < \ell$ and $m=1, \ldots , R$,
the linear system~\eqref{eq.system} has full rank.
\end{proposition}

Given Proposition~\ref{prop.freqmarching},  the system~\eqref{eq.system}
will have a unique solution for $f^\ell[r]$ once\\ $R(\ell-1) - (\lceil \ell/2 \rceil -1) \geq \dim H_\ell$. Since the invariants
of degree two determine the Gram matrices $(A^\ell)^T A^\ell$, we know $A^1 = (f^1[1], \ldots , f^1[R])$ up to the action of
$\Orth(n)$. Once we choose a representative for $A_1$ in its $\Orth(n)$-orbit, we can determine the matrices $A^2, \ldots , A^L$
inductively using Proposition~\ref{prop.freqmarching}. 

\rev{\subsection{Proof of Proposition~\ref{prop.freqmarching}}
For fixed $i,m,n$ we have the linear equation 
\begin{equation} \label{eq:linear}\sum_{k=1}^{m_\ell}C_k(f^i[m],f^{\ell-i}[n]) f^\ell_\ell[r] = I_3(i,\ell-i,\ell)(f^i[m],f^{\ell-i}[n],f^{\ell}[r]),\end{equation}
where we view $C_k(f^i[m], f^{\ell-i}[n])$ as a linear form in variables $x^i_{s,t}[m,n] = f^i_s[m]f^{\ell-i}_t[m]$. 
For distinct pairs $(m,n), (m',n') \in S_{(i,\ell-i)}[R]$, there are no algebraic relations between the sets of $m_i m_{\ell-i}$ monomials
$\{x^i_{s,t}[m,n] = f^i_s[m]f^{\ell-i}_t[n]\}_{s,t}$ and $\{x^i_{s,t}[m',n'] =f^i_s[m']f^{\ell-i}_t[n']\}_{s,t}$, so we may view them as distinct sets of variables.
Thus, we obtain an $|S_{i,\ell-i}[R]|\times m_\ell$ linear system whose coefficient matrix is
$$M_i = \begin{bmatrix}
C_1(\{x^i_{s,t}[1,1]\}) & C_2(\{x^i_{s,t}[1,1]\}) & \ldots & C_{m_\ell}(\{x^i_{s,t}[1,1]\}) \\
\ldots \\
C_1(\{x^i_{s,t}[1,R]\}) & C_2(\{x^i_{s,t}[1,R]\}) & \ldots & C_{m_\ell}(\{x^i_{s,t}[1,R]\}) \\
C_1(\{x^i_{s,t}[2,1]\}) & C_2(\{x^i_{s,t}[2,1]\}) & \ldots & C_{m_\ell}(\{x^i_{s,t}[2,1]\}) \\
\ldots\\
C_1(\{x^i_{s,t}[R,1]\}) & C_2(\{x^i_{s,t}[R,1]\}) & \ldots & C_{m_\ell}(\{x^i_{s,t}[R,1]\}) \\
\end{bmatrix}
$$
if $i \neq \ell-i$, and
$$M_i = \begin{bmatrix}
C_1(\{x^i_{s,t}[1,1]\}) & C_2(\{x^i_{s,t}[1,1]\}) & \ldots & C_{m_\ell}(\{x^i_{s,t}[1,1]\}) \\
\ldots \\
C_1(\{x^i_{s,t}[1,R]\}) & C_2(\{x^i_{s,t}[1,R]\}) & \ldots & C_{m_\ell}(\{x^i_{s,t}[1,R]\})\\
\end{bmatrix}
$$
if $i = \ell-i$. Here, $C_k(\{x^i_{s,t}[m,n]\})$ is the linear form $\sum_{s=1}^{m_i}\sum_{t=1}^{m_{\ell-i}} c_{s,t}^k x^i_{s,t}[m,n]$.
\begin{lem} \label{lem.linind}
For fixed $i, (m,n)$ the 
the forms $C_k(\{x^i_{s,t}[m,n]\})$, $i = 1, \ldots , m_\ell$ are linearly independent
over $\mathbb{R}$. 
\end{lem}
\begin{proof}{Proof of Lemma~\ref{lem.linind}.} As $f^i_s[m], f^{\ell-i}_t[n]$ vary over $H_i$ and $H_{\ell-i}$ respectively,
the $m_\ell$ vectors $C_k(f^{i}_s[m]_{s,t}[m,n]\})$ span the image of $H_i \otimes H_{\ell-i}$
in $H_\ell$. This image must be all of $H_\ell$ because if it were not then it would span a proper
$\SO(n)$-invariant subspace of the irreducible representation $H_\ell$. Since $\dim H_\ell =m_\ell$
if there were a linear relation among the the forms $C_k(\{x^{i}_{s,t}[m,n]\})$ then
the image of $H_i \otimes H_{\ell-i}$ in $H_\ell$ would not span.
\end{proof}
We now show that the  matrix $M_i$ must have full rank over
the field of rational functions $\R(\{x^i_{s,t}[m,n]\}_{s,t,m,n})$. To see this note that, for fixed $i,m,n$, the $m_\ell$ forms $C_k(\{x^i_{s,t}[m,n])$ are linearly independent we can, after a linear change of coordinates,
assume that they are an $m_\ell$ element subset of the $m_i m_{\ell-i}$ variables
$\{x^i_{s,t}[m,n]\}$. To simplify the notation we temporarily denote the $C_k(\{x^i_{s,t}[m,n])$
as $X^i[m,n]_k$ so that the matrix $M_i$ can be rewritten as
$$M_i = \begin{bmatrix}
X^i[1,1]_1 & X^i[1,1]_2 & \ldots & X^i[1,1]_{m_\ell} \\
\ldots \\
X^i[1,R]_1 & X^i[1,R]_2 & \ldots & X^i[1,R]_{m_\ell} \\
X^i[2,1]_1 & X^i[2,1]_2 & \ldots & X^i[2,1]_{m_\ell} \\
\ldots\\
X^i[R,1]_1 & X^i[R,1]_2 & \ldots & X^i[R,1]_{m_\ell} \\
\end{bmatrix}
$$
if $i \neq \ell-i$, and
$$M_i = \begin{bmatrix}
X^i[1,1]_1 & X^i[1,1]_2 & \ldots & X^i[1,1]_{m_\ell} \\
\ldots \\
X^i[1,R]_1 & X^i[1,R]_2 & \ldots & X^i[1,R]_{m_\ell} \\
\end{bmatrix}
$$
where each entry is an independent variable. Clearly no maximal minor of this matrix can vanish,
so $M_i$ has full rank over any field which contains the variables $\{X^i[m,n]_k\}$. In particular it has full
rank over the field $\R(\{x^i_{s,t}[m,n]\}_{s,t,m,n})$.

Likewise, there are no algebraic relations among the $x^i_{s,t}[m,n]$
as $i$ varies in the range $1 \leq i \leq \lfloor \ell/2 \rfloor$ so the matrix 
$$M= \begin{bmatrix}
M_1\\
M_2\\
\ldots\\
M_{\lfloor \ell/2 \rfloor}
\end{bmatrix}$$
can, after a suitable coordinate change, be expressed as a matrix where each entry 
is a distinct variable. Hence it 
also has full rank over the field of rational functions $\R(\{x^i_{s,t}[m,n]\}_{i,s,t,m,n})$.
This implies that when we substitute values for the coefficients, we obtain, for generic choices of
the functions $f^{i}[n]$, for $0 \leq i \leq \lfloor \frac{\ell}{2} \rfloor$, a maximal rank matrix.}

\begin{remark} We can also consider the $R^2 \times \lfloor k/2 \rfloor$ linear system \begin{equation} \label{eq.bigsystem}
\left\{ \sum_{k=1}^{m_\ell} C_k(f^i[m],f^{\ell-i}[n]) f^\ell_k[r] = I_3(i,\ell-i,\ell)(f^i[m],f^{\ell-i}[n],f^\ell[r]) 
\right\}_{1 \leq i \leq \lfloor \ell/2 \rfloor, (m,n) \in [1,R]^2}
\end{equation}
If we can show that~\eqref{eq.bigsystem} has full rank, then for generic $f$ we can improve the bound in Theorem~\ref{thm.main} to
$R\geq \sqrt{ \frac{\dim H_\ell}{\lfloor \ell/2 \rfloor}}$. 
To do this, it suffices to find a single set of values for the $f^j_s[m]$ for which the system has full rank. However, without explicit knowledge of the Clebsch-Gordan coefficients, this is very difficult.
\end{remark}

\section{Orbit recovery over $\SO(3)$} 
\label{sec:so3}
The invariants we obtain for general $\SO(n)$ are not given explicitly due to the difficulty of choosing natural bases 
for the irreducible representations $H_\ell$ and computing the Clebsch-Gordan coefficients. However, for $\SO(3)$ there are
explicit functional bases for the $H_\ell$ and closed formulas for the Clebsch-Gordan coefficients. For these reasons, it is possible to use a mix of computational and theoretical techniques to answer orbit separation questions.
Before we prove Theorem~\ref{thm.so3}, we recall some known results, which we use in the proof.

For $\SO(3)$,  we have an explicit decomposition of $H_i \otimes H_j$ into irredcuibles for any $i,j$. 
Precisely, we know that 
\begin{equation}
H_i\otimes H_j=\bigoplus_{\ell=|i-j|}^{i+j} H_\ell.
\end{equation}
Hence, $H_\ell$ appears with multiplicity one in any tensor product $H_i\otimes H_j$ with $i+j\geq \ell \geq i-j$.
Using a basis for $H_\ell$ of spherical harmonic functions $Y_{\ell}^k(\theta,\phi)$, we can obtain a large number of explicit invariants
of degree three:
\begin{align}
I_3(\ell_1,\ell_2,\ell_3)(f^{\ell_1}, f^{\ell_2},f^{\ell_3}) =\sum_{k_1+k_2+k_3=0,|k_i|\leq \ell_i}(-1)^{k_1}
\langle \ell_2 k_2 \ell_3 k_3|\ell_1(-k_1)\rangle f^{\ell_1}_{k_1}f^{\ell_2}_{k_2}f^{\ell_3}_{k_3},
\end{align}
where \begin{align} \label{eq.so3inv}
\langle \ell_1 \, m_1 \, \ell_2 \, m_2 | \ell \, m \rangle = \delta_{m,m_1 + m_2} \sqrt{\frac{(2\ell+1)(\ell+\ell_1-\ell_2)!(\ell-\ell_1+\ell_2)!(\ell_1+\ell_2-\ell)!}{(\ell_1+\ell_2+\ell+1)!}} \nonumber\\\;\times 
\sqrt{(\ell+m)!(\ell-m)!(\ell_1-m_1)!(\ell_1+m_1)!(\ell_2-m_2)!(\ell_2+m_2)!} \nonumber\\ \;\times 
\sum_k \frac{(-1)^k}{k!(\ell_1+\ell_2-\ell-k)!(\ell_1-m_1-k)!(\ell_2+m_2-k)!(\ell-\ell_2+m_1+k)!(\ell-\ell_1-m_2+k)!}.
\end{align}
When $\ell_3= \ell$ and $\ell_1, \ell_2 < \ell$ but $\ell_1 + \ell_2 \geq \ell$, then~\eqref{eq.so3inv} gives a linear equation for the
unknowns $f^\ell_k$ in terms of the coefficients $f^{\ell_1}_s, f^{\ell_2}_t$ in the expansions 
$f^{\ell_1} = \sum_{s=-\ell_1}^{\ell_1} f^{\ell_1}_s Y_{\ell_1}^s$ and 
$f^{\ell_2} = \sum_{t=-\ell_2}^{\ell_2} f^{\ell_2}_t Y_{\ell_2}^t$ 
of $f^{\ell_1}$ and $f^{\ell_2}$ in spherical harmonics.
In~\cite[Section 4.6.1]{bandeira2017estimation}, the authors proved, using numerical techniques, that if $R \geq 3$ then for generic $f$ the 
unknown vectors $f^\ell[r]$ for $2 \leq \ell \leq 16$ are determined from the vectors $f^1[m]$ for $m =1 , \ldots R$.
In addition, if $R \geq 3$, then invariants of degree at most three determine the 
coefficients $f^1[m]$ up to the action of $\SO(3)$ (as opposed to $\Orth(3)$). 

\begin{proof}[Proof of Theorem~\ref{thm.so3}] When $R=3$, the expression $3(\ell-1) - (\lceil \ell/2 \rceil -1) \geq 2\ell+1$ when $\ell \geq 6$.
In particular, this means  by Proposition~\ref{prop.freqmarching} we can determine $f^\ell[r]$ for $k \geq  6$ from the vectors $f^i[m]$ for $1 \leq i  < \ell$ and $1 \leq m \leq 3$.
By the results of~\cite{bandeira2017estimation}, we know that invariants of degree at most three determine the $\SO(3)$ orbit of the unknown vectors
$f^i[m]$ for $1 \leq i \leq 6$. 
\end{proof}

\begin{remark}{(The optimality of Theorem~\ref{thm.so3})}
The result of Theorem~\ref{thm.so3} is as strong as possible in the following sense. If $R < 3$, then, as previously observed in~\cite{bandeira2017estimation}, for $L$ small enough invariants
of degree at most three in $V = (\mathbf{L}^2(S^2)_L)^R$ do not generate a field of transendence degree equal to $\dim V/\SO(3)$.
In this case, a frequency marching algorithm  cannot be used to recover a generic orbit from invariants of degree at most three. 
\end{remark}

\section{Numerical Experiments}
\label{sec:numerics}

In this section, we are interested in numerically investigating the performance of the frequency marching algorithm and its performance in the presence of noise.

\subsection{Setting}
Let $f \in \mathbb{R}^3$ be a smooth, real-valued function. We assume that $f$ is bandlimited in the sense that it can be represented using a truncated 3-D spherical-Bessel expansion:
\begin{equation}
\label{eq:3-d-fourier-bessel}
f(cr, \theta, \varphi) = \sum_{\ell=0}^{\ell_\text{max}} \sum_{m=-\ell}^{\ell} \sum_{s=1}^{S(\ell)} x_{\ell, m, s} \, \mathcal{Y}_\ell^m(\theta, \varphi) \, j_{\ell, s}(r).
\end{equation}
where $c$ is the bandlimit, $S(\ell)$ is determined by the Nyquist criterion~\cite{bhamre2017anisotropic}, and $j_{\ell, s}(r)$ is the normalized spherical Bessel function:
\begin{equation}
j_{\ell, s}(r) = \frac{4}{\lvert j_{\ell + 1}(u_{\ell, s}) \rvert} \, j_\ell(u_{\ell, s} r),
\end{equation}
with $j_\ell$ the spherical Bessel function of order~$\ell$ and $u_{\ell, s}$ the $s$-th positive zero of~$j_\ell$. The real spherical harmonics $\mathcal{Y}_\ell^m$ are defined using the complex spherical harmonics:
\begin{equation}
\mathcal{Y}_\ell^m(\theta, \varphi) =
\begin{cases}
\sqrt{2} \, \Re Y_\ell^m(\theta, \varphi), & m > 0, \\
Y_\ell^0(\theta, \varphi), & m = 0, \\
\sqrt{2} \, \Im Y_\ell^{|m|}(\theta, \varphi), & m < 0,
\end{cases}
\end{equation}
where the complex spherical harmonics $Y_{\ell}^m$ are defined as
\begin{equation}
Y_{\ell}^m(\theta, \varphi) = \sqrt{\frac{2\ell+1}{4\pi} \frac{(\ell-m)!}{(\ell+m)!}} \, P_{\ell}^m(\cos \theta) \, e^{i m \varphi},
\end{equation}
where $P_{\ell}^m$ are the associated Legendre polynomials with the Condon–Shortley phase. We set $c = 1/2$ to match the Nyquist sampling rate~\cite{levin20183d}. 

The sum over $s$ in~\eqref{eq:3-d-fourier-bessel} can be interpreted as summing over concentric radial ``shells." To investigate the effect of shell count, we limit the sum to a fixed number $R$: 
\begin{equation}
\label{eq:3-d-fourier-bessel_shells}
f(cr, \theta, \varphi) = \sum_{\ell=0}^{\ell_\text{max}} \sum_{m=-\ell}^{\ell} \sum_{s=1}^{R} x_{\ell, m, s} \, \mathcal{Y}_\ell^m(\theta, \varphi) \, j_{\ell, s}(r), \quad r \le 1.
\end{equation}

The bispectrum is then defined as
\begin{equation}
\label{eq:bispectrum}
\begin{array}{ll}
B(x)[\ell_1,\ell_2,\ell_3, s_1, s_2, s_3] = & \\ 
\sum_{\substack{m_1 + m_2 + m_3 = 0\\ |m_i| \leq \ell_i}} (-1)^{m_1} \left\langle \ell_2 m_2 \, \ell_3 m_3 \mid \ell_1 (-m_1) \right\rangle x_{\ell_1, m_1, s_1} x_{\ell_2, m_2, s_2} x_{\ell_3, m_3, s_3},
\end{array}
\end{equation}
where the Clebsch–Gordan coefficients $\langle \ell_1 m_1 \, \ell_2 m_2 \mid \ell m \rangle$ are defined in~\eqref{eq.so3inv}.

\subsection{Algorithm}
\label{sec:algorithm}
\rev{
While recovering $f$ from its bispectrum is a non-convex problem, recovering the coefficients corresponding to the $\ell$-th frequency, given the coefficients of frequencies $1,\ldots,\ell-1$, is a linear problem. Thus, we recover $f$ by successively solving the linear system of equations~\eqref{eq:bispectrum}; this is the frequency marching algorithm \rev{proposed in~\cite{bandeira2017estimation}, which thanks to our Theorem~\ref{thm.so3} is guaranteed to recover generic signals from the bispectrum, up to a global rotation.}

The algorithm receives $N_\mathrm{obs}$ noisy measurements. We first expand each measurement in the spherical-Bessel basis~\eqref{eq:3-d-fourier-bessel_shells}. Let $x^{(i)}_{\ell,m,s}$ denote the spherical-Bessel coefficients of the $i$-th noisy measurement. We estimate the zeroth-frequency coefficient by the empirical first moment
\begin{equation}
\label{eq:first_moment_est}
\rev{\widehat{x}_{0,0,s}=\frac{1}{N_{\mathrm{obs}}}\sum_{i=1}^{N_{\mathrm{obs}}}x^{(i)}_{0,0,s},\qquad 1\leq s\leq R.}
\end{equation}
\rev{Similarly, the Gram matrix of the $\ell=1$ coefficients is estimated from the empirical second moment}
\begin{equation}
\label{eq:second_moment_est}
\rev{\widehat{G}_{s,t}=\frac{1}{N_{\mathrm{obs}}}\sum_{i=1}^{N_{\mathrm{obs}}}\sum_{m=-1}^{1}x^{(i)}_{1,m,s}x^{(i)}_{1,m,t},\qquad 1\leq s,t\leq R.}
\end{equation}
The required bispectrum entries are estimated by averaging~\eqref{eq:bispectrum} over the noisy measurements. The $\ell=1$ component of the signal $f$, denoted $A^1\in \mathbb{R}^{3\times R}$, can be determined up to the action of $\Orth(3)$ from the Gram matrix $(A^1)^TA^1$, estimated by~\eqref{eq:second_moment_est}. 
Now let $U\in \mathbb{R}^{3\times R}$ be a matrix with $U^TU=\widehat{G}$. We have two distinct $\SO(3)$ orbits represented by $U$ and its reflection $U'$. Either $U$ or $U'$ is in the same $\SO(3)$-orbit of $A^1$. In order to determine the correct orbit we compare the determinants of the $3\times 3$ minors of $A^1$ with both of $U$ and $U'$. These determinants can be written as $I_3(1,1,1)(f^1[r_1],f^1[r_2],f^1[r_3])$ for $1\leq r_1,r_2,r_3\leq R$ in our notation and hence they are entries of the bispectrum~\eqref{eq:bispectrum}. This fixes a representative of the recovered orbit; the final reconstruction is nevertheless determined only up to a single global rotation, as is inherent to the orbit recovery problem.}

\begin{algorithm}
\caption{Volume reconstruction from noisy measurements using frequency marching\label{alg:1}}
\begin{algorithmic}[1]
\State \textbf{Inputs:} {Noisy measurements $\{V_i\}_{i=1}^{N_\mathrm{obs}}$, the maximal frequency $\ell_{\text{max}}$, and the number of shells $R$}
\State \textbf{Outputs:} 
The spherical-Bessel coefficients $x_{\ell, m, s}$~\eqref{eq:3-d-fourier-bessel_shells}, {up to a global rotation}
\Statex

\State {Expand each noisy measurement $V_i$ in the spherical-Bessel basis~\eqref{eq:3-d-fourier-bessel_shells}.}
\State {Estimate the zeroth-frequency coefficients using~\eqref{eq:first_moment_est}.}
\State {Estimate the Gram matrix of the $\ell=1$ coefficients using~\eqref{eq:second_moment_est}.}
\State {Estimate the required bispectrum entries by averaging~\eqref{eq:bispectrum} over $\{V_i\}_{i=1}^{N_\mathrm{obs}}$.}
\State {Factor $\widehat{G}=U^TU$ and use the estimated bispectrum entries $I_3(1,1,1)$ to choose between the two $\SO(3)$ orbits represented by $U$ and its reflection.}
\For{frequency $\ell' = 2$ to $\ell_{\text{max}}$}
        \State Compute $\{x_{\ell', m, s}\}_{m=-\ell', s=1}^{m=\ell', s=R}$ by solving the linear system of equations~\eqref{eq:bispectrum} given coefficients up to {frequency $\ell'-1$}.
\EndFor
\end{algorithmic}
\end{algorithm}

\rev{\subsection{Computational complexity}
We analyze the computational complexity of Algorithm \ref{alg:1}. 
The algorithm assumes that the required bispectrum entries are available and then recovers the spherical-Bessel coefficients degree by degree by solving linear systems.

For a fixed frequency $\ell'$, the algorithm uses all admissible degree triples involving the target frequency $\ell'$ and two already recovered lower frequencies. 
The number of such lower-frequency pairs scales as $O(\ell'^2)$. 
For each target radial shell, the unknowns are the $(2\ell'+1)$ coefficients $\{x_{\ell',m,s}\}_{m=-\ell'}^{\ell'}$ at that shell. 
The equations for different target shells are decoupled, so the recovery at frequency $\ell'$ consists of $R$ independent least-squares problems. 
For each target shell, the number of equations scales as $O(\ell'^2R^2)$, where the factor~$R^2$ comes from the two lower-frequency radial indices. 
Solving each least-squares problem costs $O(\ell'^2R^2\cdot(2\ell'+1)^2)=O(\ell'^4R^2)$. 
Since there are $R$ target shells, the total recovery cost at frequency $\ell'$ is $O(\ell'^4R^3)$. 
Summing over $\ell'=2,\ldots,\ell_{\text{max}}$ gives a recovery-stage complexity of $O(R^3\ell_{\text{max}}^5)$.

In addition to this recovery stage, the required bispectrum entries must be computed. 
In the implementation, for each target frequency and target shell, we compute the bispectrum entries involving that target frequency. 
Each bispectrum entry involves a summation over spherical-harmonic orders satisfying $m_1+m_2+m_3=0$. 
Summing this angular cost over the admissible degree triples and over all target frequencies contributes a factor~$O(\ell_{\text{max}}^5)$, and the three radial indices contribute a factor~$R^3$. 
Thus, computing the required bispectrum entries from $N_{\mathrm{obs}}$ observations costs $O(N_{\mathrm{obs}} R^3\ell_{\text{max}}^5)$. Thus, 
the overall computational complexity is $O(N_{\mathrm{obs}} R^3\ell_{\text{max}}^5)$. 
}

\subsection{Numerical experiments} 
We consider \rev{three} volumes in our experiments: the TRPV1 structure~\cite{gao2016trpv1}, available from the Electron Microscopy Data Bank (EMDB) under accession code \mbox{EMD-8117}\footnote{\url{https://www.ebi.ac.uk/emdb/}}, \rev{the 20S proteasome~\cite{von2025substrates}, available as \mbox{EMD-19151},} and the Plasmodium falciparum 80S ribosome~\cite{wong2014cryo}, available as \mbox{EMD-2660}. \rev{Two of these volumes have nontrivial molecular symmetries: TRPV1 has~$C_4$ symmetry, while the 20S proteasome has~$D_7$ dihedral symmetry. These examples allow us to evaluate recovery in the presence of prescribed or approximate symmetry.} All volumes were downsampled to $31^3$ voxels and expanded using varying values of $\ell_{\text{max}}$ and $R$. The experiments were conducted on a MacBook Pro (2023) equipped with an Apple M3 Pro chip (12-core CPU, 18-core GPU) and 36GB of unified memory, running macOS 15.5.
The code used to reproduce all numerical experiments is publicly available at 
\url{https://github.com/krshay/orbit-recovery-for-spherical-functions}.

\textbf{Recovery in the absence of noise.} We begin by demonstrating successful volume reconstructions from clean bispectrum data; see Figure~\ref{fig:reconstructions}. These results use $\ell_{\text{max}} = 10$ and $R = 8$ for both volumes. The code was executed natively without virtualization, and took approximately 17 seconds for the reconstruction. The molecular visualizations were produced using UCSF Chimera~\cite{pettersen2004ucsf}. We emphasize that the algorithm accurately recovers the volume for $R \geq 3$, as expected from the theory, and we provide an example with additional shells for visual clarity.

\begin{figure}[h]
  \centering
  \begin{subfigure}{0.32\textwidth}
    \includegraphics[width=\linewidth]{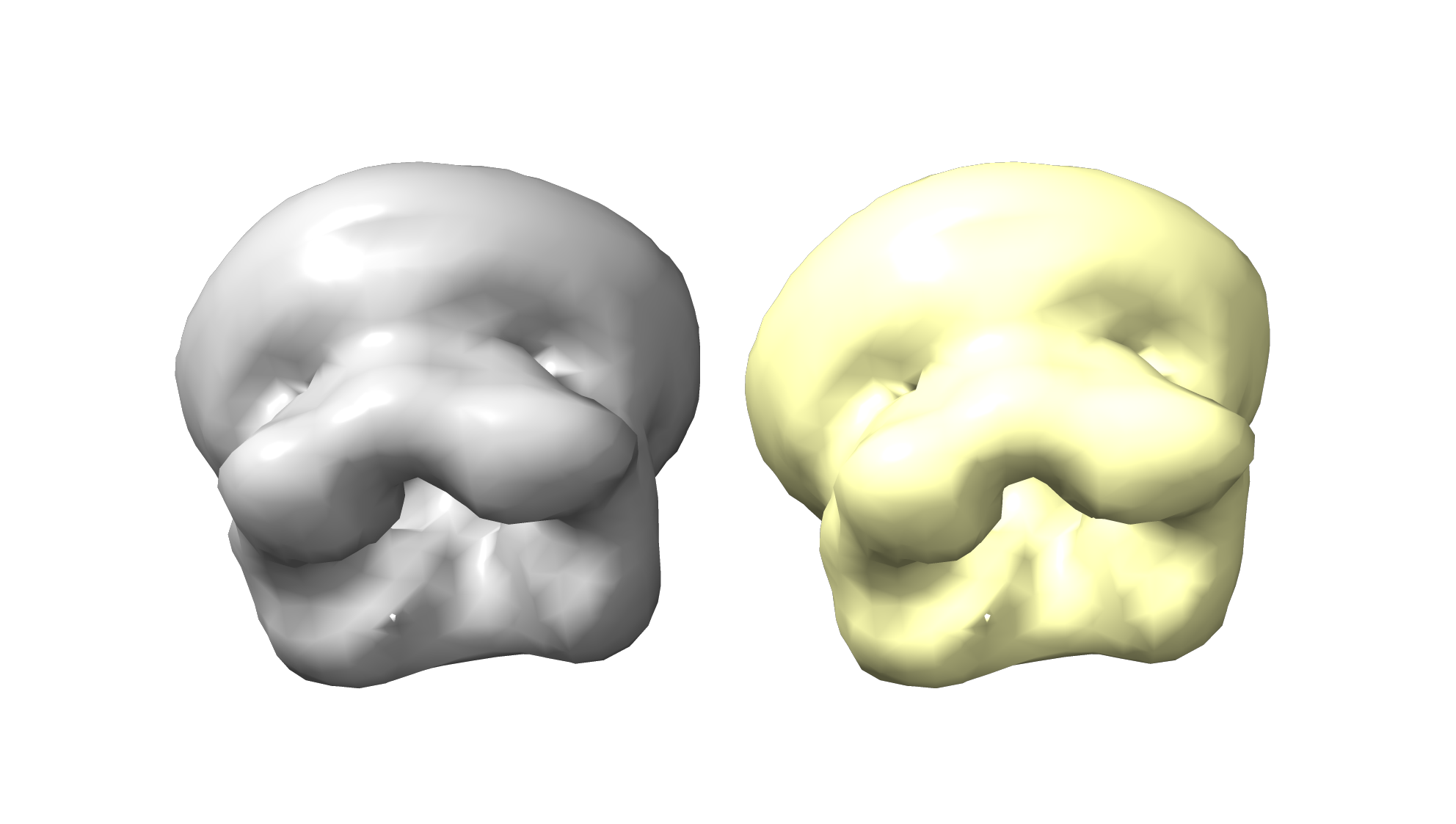}
    \caption{The TRPV1 structure}
    \label{fig:fig1}
  \end{subfigure}
  \hfill
  \begin{subfigure}{0.32\textwidth}
    \includegraphics[width=\linewidth]{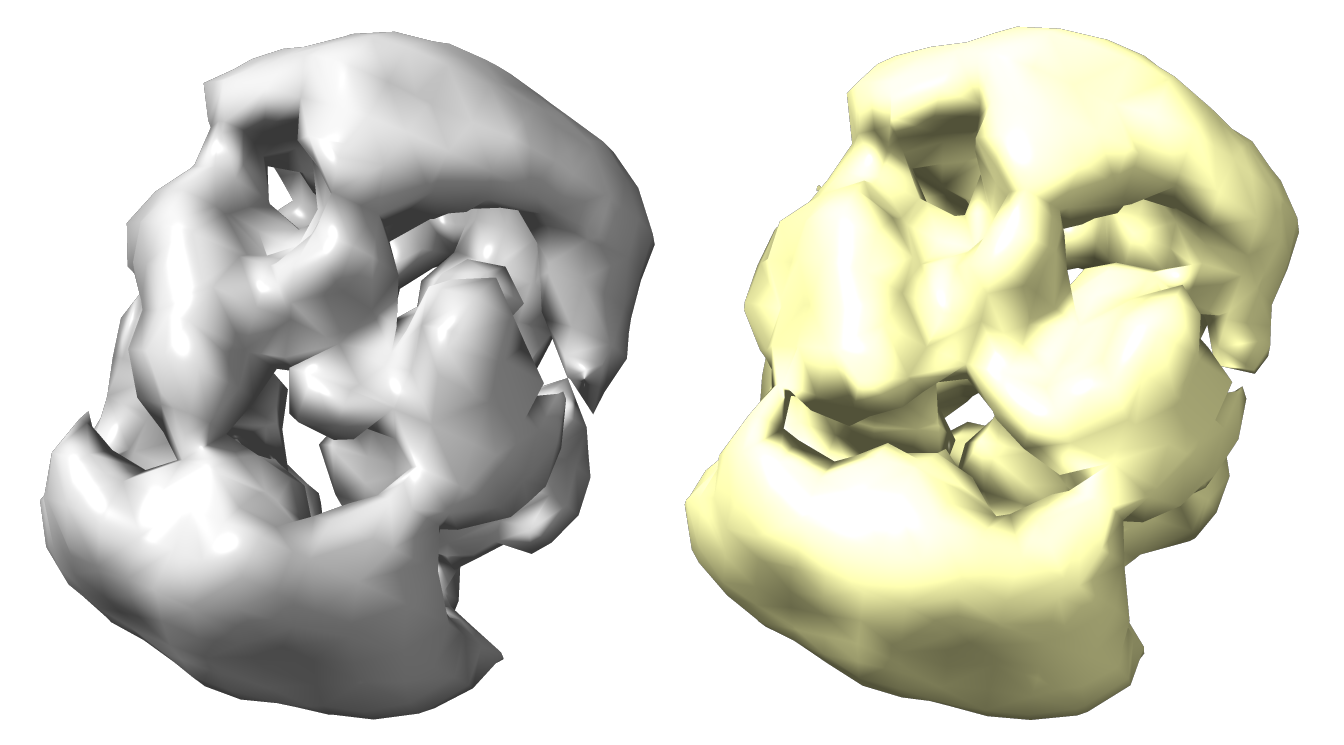}
    \caption{\rev{The 20S proteasome}}
    \label{fig:fig-d7}
  \end{subfigure}
  \hfill
  \begin{subfigure}{0.32\textwidth}
    \includegraphics[width=\linewidth]{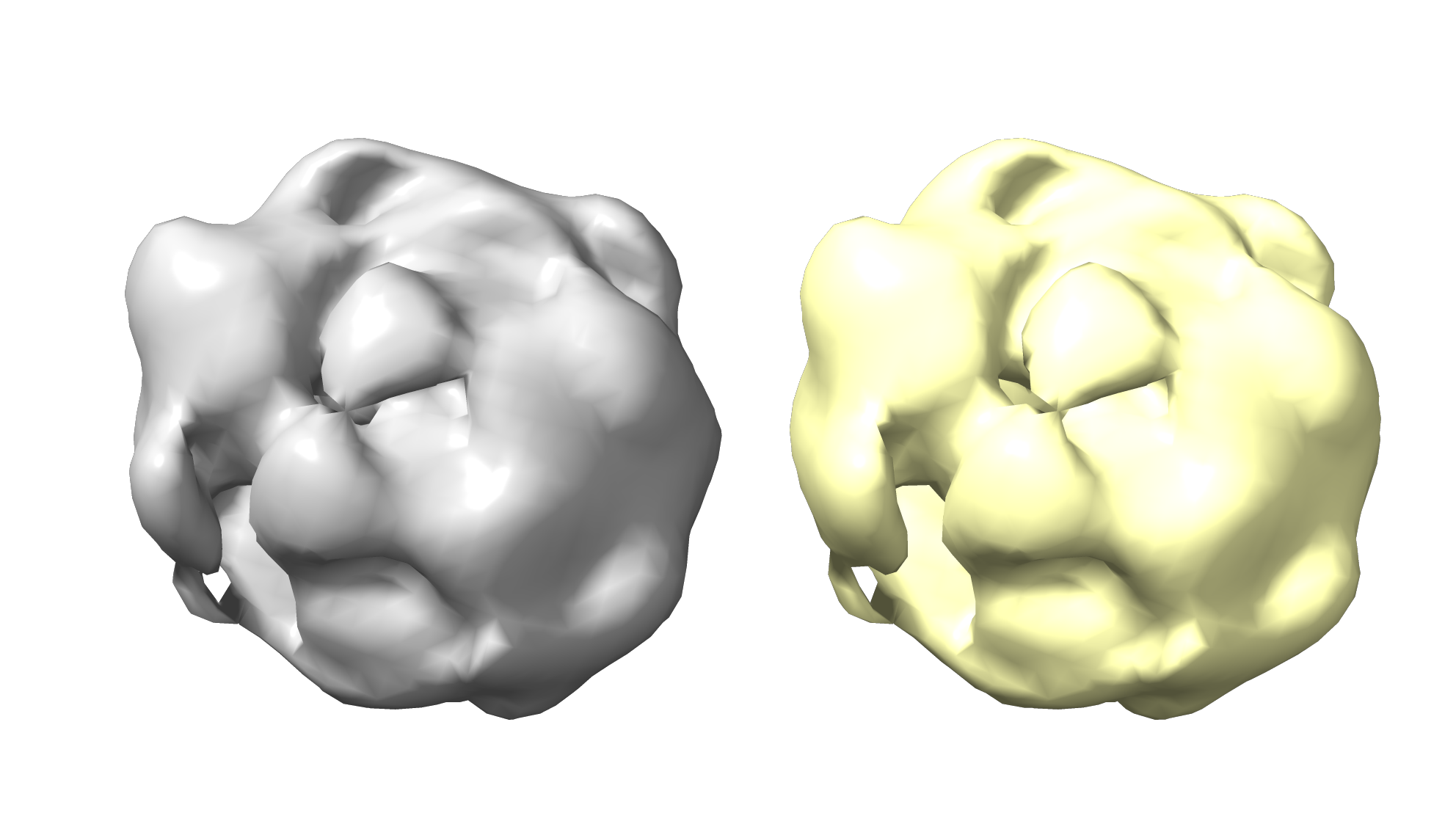}
    \caption{The 80S ribosome}
    \label{fig:fig2}
  \end{subfigure}
  \caption{Volume reconstructions from clean bispectrum, with $\ell_{\text{max}} = 10$ and $R = 8$ for the TRPV1 structure\rev{, the 20S proteasome}, and the 80S ribosome. Left (gray): ground truth after expanding to the corresponding maximal frequency $\ell_{\text{max}}$; right (yellow): reconstruction.}
  \label{fig:reconstructions}
\end{figure}

\textbf{Condition number analysis.} In practice, it is important that the recursive system of linear equations has good condition numbers in order to ensure the stability of the solutions. Fortunately, our numerical analysis indicates that these systems are generally well-conditioned. Tables~\ref{tab:trpv1-cond}, \rev{~\ref{tab:d7-cond}}, and \ref{tab:80s-cond}, and report the condition number. \rev{Overall, the condition numbers are moderate across all tested settings and generally improve with frequency and with the number of radial shells. This indicates that the recursive systems are well-conditioned in practice.}
\rev{The running times for producing these condition-number tables were approximately 2.4, 3.7, and 6.2 seconds for TRPV1 with $R=3,4,5$, respectively; approximately 2.6, 4.0, and 6.3 seconds for the 20S proteasome with $R=3,4,5$, respectively; and approximately 2.6, 4, and 5.9 seconds for the 80S ribosome with $R=3,4,5$, respectively.}

\begin{table}[h]
  \centering
  \caption{Condition numbers for TRPV1}
  \label{tab:trpv1-cond}
  \begin{tabular}{c|ccc}
    \toprule
    $\ell$ & 3 shells & 4 shells & 5 shells \\
    \midrule
    2  & 8749 & 1391 & 609.9 \\
    3  & 247.9 & 115.0 & 25.12 \\
    4  & 27.32 & 15.11 & 15.65 \\
    5  & 17.52 & 11.98 & 11.65 \\
    6  & 19.36 & 11.68 & 10.96 \\
    7  & 21.64 & 16.44 & 13.85 \\
    8  & 15.38 & 10.92 & 9.707 \\
    9  & 14.52 & 11.29 & 10.56 \\
    10 & 15.24 & 11.35 & 10.63 \\
    \bottomrule
  \end{tabular}
\end{table}

\begin{table}[h]
  \centering
  \caption{Condition numbers for the 20S proteasome}
  \label{tab:d7-cond}
  \begin{tabular}{c|ccc}
    \toprule
    $\ell$ & 3 shells & 4 shells & 5 shells \\
    \midrule
    2  & 590.9 & 109.4 & 11.15 \\
    3  & 380.4 & 37.97 & 10.41 \\
    4  & 229.1 & 108.8 & 39.15 \\
    5  & 57.22 & 32.47 & 18.12 \\
    6  & 41.12 & 30.80 & 13.44 \\
    7  & 29.46 & 18.75 & 10.85 \\
    8  & 99.47 & 30.92 & 12.05 \\
    9  & 22.21 & 13.00 & 7.410 \\
    10 & 26.57 & 13.06 & 8.435 \\
    \bottomrule
  \end{tabular}
\end{table}

\begin{table}[h]
  \centering
  \caption{Condition numbers for the 80S ribosome}
  \label{tab:80s-cond}
  \begin{tabular}{c|ccc}
    \toprule
    $\ell$ & 3 shells & 4 shells & 5 shells \\
    \midrule
    2  & 3.037 & 3.209 & 3.384 \\
    3  & 7.301 & 3.213 & 3.138 \\
    4  & 3.309 & 2.593 & 2.614 \\
    5  & 3.835 & 2.792 & 2.339 \\
    6  & 2.829 & 2.105 & 1.705 \\
    7  & 2.665 & 2.163 & 1.875 \\
    8  & 2.811 & 1.948 & 1.771 \\
    9  & 2.774 & 2.210 & 2.138 \\
    10 & 2.217 & 2.004 & 1.995 \\
    \bottomrule
  \end{tabular}
\end{table}

\newpage

\textbf{Robustness to Noise}. In practical settings, the bispectrum is not available in its exact form but rather estimated from noisy observations. To evaluate the algorithm's robustness to noise, we consider the following experimental setup.
Let $f_{\text{true}}$ denote the ground-truth volume---in this case, the Plasmodium falciparum 80S ribosome. We simulate noisy measurements according to the model $ V_i = f_{\text{true}} + \varepsilon_i $, where each $ \varepsilon_i $ is a Gaussian noise vector with zero mean and variance $\sigma^2=0.5$. 
A total of 500 noisy measurements $ V_i $ are generated.
For each noisy measurement, we compute its bispectrum and average the results across all 500 instances. The coefficient vector is then estimated using Algorithm~\ref{alg:1}; \rev{the coefficients corresponding to frequencies $\ell=0$ and $\ell=1$ are assumed to be known.}. The experiment is performed with a bandlimit of $ \ell_{\text{max}} = 10 $, and the number of radial shells $ R $ is varied from 3 to 8. The resulting recovery errors are presented in Figure~\ref{fig:err_num_shells}.
The recovery error is computed as
\begin{equation}
\text{Relative Error} = \frac{\|x - \hat{x}\|_\text{F}}{\|x\|_\text{F}},
\end{equation}
where $x$ denotes the ground-truth expansion coefficients, $\hat{x}$ the recovered coefficients, and $\|\cdot\|_{\text{F}}$ the Frobenius norm.
Interestingly, while theoretical guarantees suggest that 3 shells suffice for accurate reconstruction, the results indicate that using more than 3 shells yields improved robustness in the presence of noise.

\begin{figure}[h]
\centering
	\includegraphics[width=0.7\columnwidth, keepaspectratio]{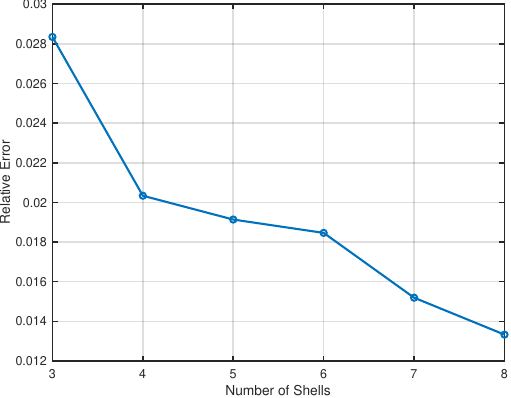}
	\caption{Recovery error of the 80S ribosome as a function of the number of shells, under additive Gaussian noise.}
	\label{fig:err_num_shells}
\end{figure}

\bibliographystyle{siamplain}

\end{document}